\documentclass[10pt]{scrartcl}

\usepackage[utf8]{inputenc}
\usepackage[T1]{fontenc}

\usepackage{amsmath, amsthm, amssymb}

\usepackage[a4paper]{geometry}

\usepackage[unicode]{hyperref}
\hypersetup{%
  colorlinks=true,
  linkcolor=black,
  citecolor=black,
  urlcolor=blue,
  pdftitle={},
  pdfauthor={J. Lankeit},
  pdfkeywords={},
  bookmarksopen=true,
}

\newcommand{\nn}{\nonumber}
\newcommand{\Om}{\Omega} 
\newcommand{\Ombar}{\overline{\Om}}
\newcommand{\uet}{u_{εt}}
\newcommand{\vet}{v_{εt}}
\newcommand{\ue}{u_{ε}}
\newcommand{\ve}{v_{ε}}
\newcommand{\une}{u_{0ε}}
\newcommand{\vne}{v_{0ε}}
\newcommand{\ye}{y_{ε}}
\newcommand{\f}[2]{\frac{#1}{#2}}
\newcommand{\norm}[2][]{\|#2\|_{#1}}
\newcommand{\normm}[2]{\|#2\|_{#1}}
\newcommand{\Lom}[1]{L^{#1}(\Om)}
\newcommand{\Wom}[2]{W^{#1,#2}(\Om)}
\newcommand{\Womstar}[2]{(W^{#1,#2}_0(\Om))^*}
\newcommand{\set}[1]{\{#1\}}
\newcommand{\lset}[1]{\left\{#1\right\}}
\newcommand{\setl}[1]{\left\{#1\right\}}
\newcommand{\att}{(\cdot,t)}
\newcommand{\ats}{(\cdot,s)}
\newcommand{\io}{\int_{\Om}}
\newcommand{\kl}[1]{\left(#1\right)}
\newcommand{\Loneloc}{L^1_{loc}([0,∞))}
\newcommand{\ddt}{\frac{\mathrm{d}}{\mathrm{d}t}}
\usepackage{nicefrac}
\newcommand{\nf}[2]{\nicefrac{#1}{#2}}

\newcommand{\ds}{\mathrm{d}s}
\newcommand{\dsigma}{\mathrm{d}σ}
\newcommand{\dy}{\mathrm{d}y}
\newcommand{\dz}{\mathrm{d}z}
\newcommand{\wto}{\rightharpoonup}

\DeclareMathOperator{\conv}{conv}

\usepackage{newunicodechar}
\newunicodechar{∇}{\nabla}
\newunicodechar{Δ}{\Delta}
\newunicodechar{ℝ}{\mathbb{R}}
\newunicodechar{κ}{\kappa}
\newunicodechar{λ}{\lambda}
\newunicodechar{μ}{\mu}
\newunicodechar{ε}{\varepsilon}
\newunicodechar{Φ}{\Phi}
\newunicodechar{φ}{\varphi}
\newunicodechar{χ}{\chi}
\newunicodechar{∞}{\infty}
\newunicodechar{∂}{\partial}
\newunicodechar{ν}{\nu}
\newunicodechar{→}{\to}
\newunicodechar{ℕ}{\mathbb{N}}
\newunicodechar{ϕ}{\phi}
\newunicodechar{α}{\alpha}
\newunicodechar{ζ}{\zeta}
\newunicodechar{δ}{\delta}
\newunicodechar{ρ}{\rho}
\newunicodechar{ψ}{\psi}
\newunicodechar{Ψ}{\Psi}
\newunicodechar{η}{\eta}
\newunicodechar{τ}{\tau}
\newunicodechar{σ}{\sigma}
\newunicodechar{π}{\pi}

\newtheorem{lemma}{Lemma}[section]
\newtheorem{theorem}[lemma]{Theorem}
\newtheorem{remark}{Remark}[section]

\newcounter{globalconst}
\newcounter{localconst}[lemma]

\makeatletter
\newcommand{\newgc}[1]{\refstepcounter{globalconst}\ltx@label{#1}C_{\ref{#1}}}
\newcommand{\newlc}[2][]{\refstepcounter{localconst}\ltx@label{lc:\arabic{lemma}:#2}c_{\ref{lc:\arabic{lemma}:#2}#1}}
\makeatother
\newcommand{\gc}[1]{C_{\ref{#1}}}

\newcommand{\lc}[2][]{c_{\ref{lc:\arabic{lemma}:#2}#1}}

\usepackage{xcolor}

\title{Immediate smoothing and global solutions for initial data in $L^1\times W^{1,2}$ in a Keller--Segel system with logistic terms in 2D}
\author{Johannes Lankeit\footnote{Institut für Mathematik, Universität Paderborn, Warburger Str. 100, 33098 Paderborn, Germany; e-mail: jlankeit@math.uni-paderborn.de}}

\begin{document}

\maketitle 
\begin{abstract}
\noindent
 This article deals with the logistic Keller--Segel model 
\[
 \begin{cases}
u_t = \Delta u - \chi \nabla\cdot(u\nabla v) + \kappa u - \mu u^2, \\
v_t = \Delta v - v + u   
 \end{cases}
\]
 in bounded two-dimensional domains (with homogeneous Neumann boundary conditions and for parameters $\chi, \kappa\in \mathbb{R}$ and $\mu>0$), and shows that any nonnegative initial data $(u_0,v_0)\in L^1\times W^{1,2}$ lead to global solutions that are smooth in $\bar{\Omega}\times(0,\infty)$. 
\textbf{Math Subject Classification (MSC2020):} 35B65, 35K45, 35A09, 35Q92, 92C17
\end{abstract}
 

\section{Introduction}

%

Chemotaxis systems \cite{KS,horstmann,BBTW} are mainly known for their ability to produce singularities in the form of blow-up (see \cite{JL,HV,mizo_win} or the recent survey \cite{lan_win_survey} for an overview), that is, initially smooth solutions cease to exist and thereby lose their regularity -- for example, converging to certain multiples of Dirac measures plus an $L^1$-function as time tends to some finite blow-up time, cf. \cite{herrero.velazquez_Singularity,nagai_blowup}. This article is concerned with the opposite setting: Given initial data $(u_0,v_0)$ that are significantly less regular than smooth functions (we will strive for $L^1$ in the first component), we will ask whether the system  
\begin{subequations}\label{thesystem}
\begin{align}
 u_t &= Δu - χ ∇\cdot(u∇v) +κu-μu^2&&\qquad \text{in } \Om\times(0,∞),\label{thesystem:u}\\
 v_t &= Δv - v + u &&\qquad \text{in } \Om\times(0,∞),\label{thesystem:v}\\
 &\qquad  ∂_{ν} u = 0 = ∂_{ν} v&&\qquad \text{on } ∂\Om\times(0,∞),\label{thesystem:bdry}\\
 &\qquad u(\cdot,0) = u_0, \; v(\cdot,0)=v_0&&\qquad \text{in } \Om,\label{thesystem:initial}
\end{align}
\end{subequations}

posed in a bounded domain $\Om\subset ℝ^2$, admits solutions that immediately become classically smooth (and answer this question affirmatively, see Theorem \ref{thm:main} below).\\

System \eqref{thesystem}, combining logistic source terms as the simplest form of population dynamics and chemotaxis -- that is, the partially directed movement in response to a chemical signal --, arises in different contexts in mathematical biology, ranging from the study of cancer \cite{cancer} to ecology \cite{ecology}. 
From a modelling perspective, 
it seems important to understand all effects emerging from the interplay of taxis and population growth terms, so as to guide (and serve as solid foundation for) the creation of further and more refined models.\\

In the classical Keller--Segel system, that is \eqref{thesystem} with $κ=μ=0$, in one-dimensional domains global bounded solutions exist for any reasonably regular initial data and the result is similar in 2D, as long as the initial mass is sufficiently small in the sense that $\io u_0<\f{4π}{χ}$. Contrastingly, in 2D domains for any $m>\f{4π}{χ}$ and in 3D for any $m>0$, one can find smooth initial data $u_0$ with $\io u_0=m$ such that the corresponding solution blows up in finite time (cf. Section 2.1 of the survey  \cite{lan_win_survey} and results cited therein, in particular \cite{horstmann_wang}).

A typical assumption on the initial data -- apart from their nonnegativity -- is that they be continuous (see, e.g., \cite[Lemma 3.1]{BBTW}). 
This requirement is necessary if one wishes to obtain classical solutions under any definition involving their continuity up to time $0$, inclusively. 
Other than that, weaker assumptions can suffice: The system has been studied, for example, with initial data in certain Besov spaces, leading to solutions that are continuous as functions with values in said spaces.

Even more extensive results concerning local solvability were achieved by Biler in \cite{biler_local_and_global} already: For two-dimensional domains, the admissible class of initial data covers all of $\Lom1$ and even some measures beyond; these solutions are mild solutions and their initial data are attained as weak-*-limit in the space $\mathcal{M}(\Om)$ of finite signed  measures. The widest range of initial data for higher-dimensional settings is treated in \cite{zhigun}: nonnegative $\Lom1$-data lead to global generalized solutions (for the price of a solution concept that is based on a complicated integral inequality for a combination of both solution components and that involves rather weak regularity information).\\

%
%

However, the solutions arising in these works 
may (and for some initial data: do; compare also the above-mentioned results on blow-up even of classical solutions) still undergo blow-up. \\

On the other hand, logistic terms (i.e. $κ\in ℝ, μ>0$ in \eqref{thesystem}), are known to have a regularizing effect in chemotaxis systems. For example, in two-dimensional domains, \eqref{thesystem} admits global classical, bounded solutions (see \cite{osaki_yagi} and \cite{tian_sublogistic} or, for the analogous parabolic--elliptic system, \cite{tello_win}); similarly, in higher-dimensional domains large values of $μ$ ensure the global existence of classical solutions (\cite{win_log_boundedness}, also \cite{nakaguchi_osaki} or \cite{tian_mu_zero}). In the case of small $μ>0$ (and initial data in $\Lom2\times \Wom12$) at least weak solutions exist globally, \cite{lan_ev_smooth}, and, moreover, become smooth after some (possibly large) time if $κ$ is small, $\Om$ is three-dimensional and convex, \cite{lan_ev_smooth}, see also \cite{giuseppe_boundedness,viglialoro_woolley}.
Nevertheless, weaker degradation (at least in parabolic-elliptic relatives of \eqref{thesystem}, cf. \cite{win_log_bu,win_log_bu_ZAMP}) or spatial inhomogeneity of appropriate form in the zero-order terms, \cite{fuest}, can still facilitate blow-up. 
 
Systems like \eqref{thesystem} are covered by the general existence result of \cite{win_superlinear_damping} even if $+κu-μu^2$ is replaced by a weaker absorption $+κu-μu^{α}$ with, merely, $α>\f{2n+4}{n+4}$: For nonnegative $(u_0,v_0)\in \Lom1\times\Lom2$, generalized solutions exist.

However, in order to allow for its wide range of applicability (including said subquadratic terms and not restricted to chemotaxis systems)
, the solution concept pursued in \cite{win_superlinear_damping} is very weak: The first component, $u$, for instance, is supposed to satisfy the weak form of a differential inequality for $\ln(1+u)$ and to belong to $L_{loc}^\infty([0,∞);\Lom1)\cap L^{α}_{loc}(\Ombar\times[0,∞))$. Under more restrictive conditions on $α$ and $v_0$ -- the borderline case in $n=2$, in fact, coinciding with the case considered in the present article --, \cite{win_superlinear_damping} then goes on to assert more regular solutions, so that $u$ additionally belongs to $L^1_{loc}([0,∞);\Wom11)$ and solves the equation in the standard weak sense.\\

The results on global existence of classical solutions reported above make us suspect that at least in the prototypical system with logistic population dynamics 
and a two-dimensional domain, even classical solutions should be achievable, notwithstanding initial data merely belonging to a space of less regular than continuous functions. 
If $κ=μ=0$ and the domain is one-dimensional (one-dimensionality is a way to prevent blow-up in the Keller--Segel system), the system can cope with initial data from $(C(\Ombar))^*\times \Lom2$ and turn them into classical solutions,  \cite{win_instant_smoothing}.

For a simplified parabolic--elliptic variant of \eqref{thesystem} (with \eqref{thesystem:v} reading $0=Δv - m(t)+u$, $m(t)=\io u(\cdot,t)$) in a radially symmetric setting, classical solutions were found for initial data lying below, essentially, a multiple of $\exp(λ|x|^{-α})$ with some $λ>0$ and $α<n$, \cite{win_regularized}. \\

For the possibly non-radial setting in the fully parabolic system \eqref{thesystem}, we will assume that the initial data comply with the requirements that
\begin{equation}\label{init}
 u_0 \in L^1(\Om),\quad u_0\ge 0, \qquad v_0\in W^{1,2}(\Om), \quad v_0\ge 0
\end{equation}
and show that then solutions result, which are immediately smooth (in the sense that they belong to $C^{2,1}(\Ombar\times(0,∞))$) and continuous as $\Lom1$-valued functions (in the weak topology of $\Lom1$) up to the initial time, i.e. belong to $C_w([0,∞);\Lom1)=C([0,∞);(\Lom1 \text{with weak topology}))$:

\begin{theorem}\label{thm:main}
 Let $\Om\subset ℝ^2$ be a bounded smooth domain and assume that $u_0, v_0$ satisfy \eqref{init}. Then there is 
 \[
  (u,v)\in 
  \kl{C^{2,1}(\Ombar\times(0,∞))\cap C_w([0,∞);\Lom1)}^2
 \]
 solving \eqref{thesystem}.
\end{theorem}

The proof will proceed by approximation of these solutions by solutions to problems where global classical solvability is known, which in particular involves approximation of the initial data. 

\begin{subequations}\label{init-eps}
Given $u_0$ and $v_0$ as in \eqref{init} and $q_0>2$, for every $ε>0$ we introduce $\une$ and $\vne$ such that the following hold true: 
\begin{align}
 &\une \in C(\Ombar),\quad \vne\in W^{1,q_0}(\Om),
\label{init-eps:regularity}\\
 &\une \ge 0, \qquad \vne\ge 0\qquad \text{in } \Om,\label{init-eps:nonnegativity}\\
 &\une \to u_0 \quad \text{in } L^1(\Om) \quad \text{ as } ε→0,  \label{init-eps:conv:u}\\
 &\vne \to v_0 \quad \text{in } W^{1,2}(\Om) \quad \text{ as } ε→0, \label{init-eps:conv:v}\\
 &\norm[\Lom1]{\une}\le 2\norm[\Lom1]{u_0} \qquad \text{ for all } ε>0,\label{init-eps:bound:u}\\
 &\norm[\Wom12]{\vne}\le 2\norm[\Wom12]{v_0} \qquad \text{ for all } ε>0, \label{init-eps:bound:v}
\end{align}
\end{subequations}
and deal with solutions of the initial boundary value problem 
\begin{subequations}\label{approximatesystem}
\begin{align}
 \uet &= Δ\ue - χ∇\cdot(\ue∇\ve) + κ\ue - μ\ue^2&&\qquad \text{in } \Om\times(0,∞),\label{approximatesystem:u}\\
 \vet &= Δ\ve - \ve + \f{\ue}{1+ε\ue}&&\qquad \text{in } \Om\times(0,∞),\label{approximatesystem:v}\\
 &\qquad  ∂_{ν} \ue = 0 = ∂_{ν} \ve&&\qquad \text{on } ∂\Om\times(0,∞),\label{approximatesystem:bdry}\\
 &\qquad \ue(\cdot,0) = \une, \; \ve(\cdot,0)=\vne &&\qquad \text{in } \Om.\label{approximatesystem:initial}
\end{align}
\end{subequations}
\begin{remark}
 The approximation used here coincides with that in \cite[Sec. 4]{win_superlinear_damping}.
 In \eqref{approximatesystem:v}, we could, instead, replace $\f{\ue}{1+ε\ue}$ by $\ue$. Then global existence of solutions (cf. Lemma~\ref{lem:ex}) might be less obvious, but can, nevertheless, be gained from \cite{nakaguchi_osaki}. Otherwise, the proofs below remain essentially unaffected.
\end{remark}

For the proof of smoothness, it will suffice to deduce boundedness of the solutions in some $\Lom p$ space with sufficiently large $p$, rendering classical parabolic regularity theory applicable (cf. Lemmata \ref{lem:uinfty} to \ref{lem:uC2}). 
The process of obtaining such boundedness information will be based on a study of $\ddt\io u^2$, which due to $\io u_0^2$ possibly being infinite cannot be extended to time-intervals reaching zero, but thanks to the superlinear degradation term nevertheless provides estimates on intervals bounded away from the initial time (see Lemma \ref{lem:up}).\\

What remains is the somewhat separate investigation of the solution near time $t=0$, which has to ensure that the initial data are attained in a reasonable way. This is where the main focus of this article lies. A basic idea concerning continuity at $t=0$ would be to look for a uniform bound of $\ue(\cdot,t)$ on some interval $[0,t_0)$ and for $ε>0$, observe that this bound can be transferred to $u(\cdot,t)$, $t\in[0,t_0)$, via the approximation procedure, hope for a corresponding compactness result (maybe in some weak topology) and conclude for every sequence $(t_k)_{k\in ℕ}$, $t_k\to 0$ as $k\to \infty$ that $u(\cdot,t_k)$ converges, at least along a subsequence. 

It seems, however, that the derivation of these a priori bounds cannot be based on time derivatives of functionals like $\io u^2(\cdot,t)$ (as discussed above); even for $\io u\log u$ -- a term that is part of the helpful and often-employed Lyapunov functional for the Keller--Segel system -- the regularity of $u_0$ in \eqref{init} is clearly insufficient. 

However, as long as $u_0\in \Lom1$, there is some function $Φ$ of superlinear growth such that also $Φ(u_0)$ belongs to $\Lom1$. We will use this together with a computation of $\ddt \io Φ(u(\cdot,t))$ to finally achieve weak compactness of $\set{u(\cdot,t)\mid t\in[0,t_0)}$ in $\Lom1$. We will dedicate Section~\ref{sec:initreg} to this 
strategy, while Subsection~\ref{sec:identify}, in particular, ensures that the limit as $t\to 0$ coincides with $u_0$. 
Concerning the choice of $Φ$ (or rather the properties of $Φ$ that will be used in the proof), see also the discussion at the beginning of Section~\ref{sec:initreg}. \\

\textbf{Notation.} 
We will assume the smooth, bounded domain $\Omega\subset ℝ^2$, the parameters $κ\in ℝ$, $χ\in ℝ$, $μ>0$, initial data $(u_0,v_0)$ as in \eqref{init}, the exponent $q_0>2$ and one family $\set{(\une,\vne) | ε>0}$ obeying \eqref{init-eps} to be fixed throughout the article. Keep in mind that this, in particular, means that all constants mentioned below may depend on these objects. Constants denoted by $C_i$ with capital letter C are unique within the article (and will usually have been introduced in an earlier lemma), constants $c_i$ denoted by lowercase c are local to each proof. 
For $ε>0$, $(\ue,\ve)$ will always denote the corresponding solution of \eqref{approximatesystem}. 

\section{Existence and first \textit{a priori} estimates}
\begin{lemma}\label{lem:ex}
 For every $ε>0$, \eqref{approximatesystem} has a unique classical nonnegative solution $(\ue,\ve)$ with 
\begin{align*}
 \ue &\in C(\Ombar\times[0,\infty))\cap C^{2,1}(\Ombar\times(0,∞)) \\
 \ve &\in C(\Ombar\times[0,\infty))\cap C^{2,1}(\Ombar\times(0,∞))\cap L^{∞}_{loc}([0,∞);\Wom1{q_0}).  
\end{align*}
\end{lemma}
\begin{proof}
 Local existence and uniqueness are obtained as usual (see \cite[Lemma~3.1]{BBTW}); boundedness, and hence globality, results from the obvious boundedness of the force term $+\f{\ue}{1+ε\ue}$ in the second equation (for fixed $ε>0$).
\end{proof}

\begin{lemma}\label{lem:mass} 
 There is $\newgc{mass}>0$ such that for every $ε>0$ and every $t\in[0,∞)$,
\[
 \io \ue\att \le \gc{mass}.
\]
\end{lemma}
\begin{proof}
 According to the Cauchy--Schwarz inequality, $(\io\ue\att)^2 \le |\Om|\io \ue^2\att$ for all $t\in(0,∞)$ and all $ε>0$. 
 By integration of \eqref{approximatesystem:u}, the functions defined by $\ye(t):=\io \ue\att$, $t\in(0,∞)$, for $ε>0$ therefore satisfy 
\[
 \ye'(t) \le κ \ye(t) - \f{μ}{|\Om|} (\ye(t))^2 \qquad \text{for all } t\in(0,∞),
\]
as well as $\ye(0)\le 2\norm[\Lom1]{u_0}$ by \eqref{init-eps:bound:u},  and their boundedness follows from an ODE comparison argument. 
\end{proof}

\begin{lemma}\label{lem:u2xt}
 There is $\newgc{u2xt}>0$ such that for every $ε>0$ and every $t\in(0,∞)$, 
 \begin{equation}\label{bound-usqare-xt}
  \int_0^{t}\io \ue^2 \le \gc{u2xt}(1+t).
 \end{equation}
\end{lemma}
\begin{proof}
 Integrating \eqref{approximatesystem:u} over $\Om\times (0,t)$ we see that 
 \[
  μ\int_0^t\io \ue^2 \le κ\int_0^t\io \ue + \io \une - \io \ue(\cdot,t) \qquad \text{for every $t>0$ and $ε>0$,}
 \]
  which due to Lemma~\ref{lem:mass} and \eqref{init-eps:bound:u} yields \eqref{bound-usqare-xt} with $\gc{u2xt}:=\f1{μ}
 \max\set{\gc{mass}κ,2\norm[\Lom1]{u_0}}$.
\end{proof}

\begin{lemma}\label{lem:vl2bound}
 There is $\newgc{v2}>0$ such that for every $ε>0$ and every $t\in[0,∞)$,
 \begin{equation}\label{bound-vsquare}
  \io \ve^2\att \le \gc{v2}(1+t) \qquad \text{and}\qquad \int_0^t\io \ve^2 \le \gc{v2}(1+t).
 \end{equation}
\end{lemma}
\begin{proof}
 Testing \eqref{approximatesystem:v} by $\ve$ and applying Young's inequality, we see that 
 \[
  \ddt \io \ve^2+\io \ve^2 = -2\io |∇\ve|^2 - \io \ve^2 + 2\io \f{\ue\ve}{1+ε\ue}\le \io \ue^2 \qquad\text{on } (0,∞),
 \]
 so that \eqref{bound-vsquare} follows from Lemma~\ref{lem:u2xt} and \eqref{init-eps:bound:v} with $\gc{v2}:=\gc{u2xt}+4\norm[\Wom12]{v_0}^2$.
\end{proof}

\begin{lemma}\label{lem:vderivativebounds}
 There are $\newgc{nav2}>0$ and $\newgc{deltav2xt}>0$ such that for every $ε>0$ and every $t\in(0,∞)$,
 \begin{equation}\label{bound-navsquare-and-deltavsquare-xt}
  \io |∇\ve\att|^2 \le \gc{nav2}(1+t), \qquad \int_0^t\io |Δ\ve|^2\le \gc{deltav2xt}(1+t).
 \end{equation}
\end{lemma}
\begin{proof}
On account of Young's inequality, an integration of \eqref{approximatesystem:v} over $\Om$ results in 
 \begin{align*}
  \ddt \io |∇\ve|^2 +\io |Δ\ve|^2 &= -\io |Δ\ve|^2 -2\io |∇\ve|^2 +2\io \f{\ue}{1+ε\ue} Δ\ve
  \le  \io \ue^2
 \end{align*}
 on $(0,∞)$ for every $ε>0$. 
 Setting $\gc{nav2}:=\gc{deltav2xt}:=\gc{u2xt}+4\norm[\Wom12]{v_0}^2$,  we can rely on Lemma~\ref{lem:u2xt} and \eqref{init-eps:bound:v} to derive \eqref{bound-navsquare-and-deltavsquare-xt}.
\end{proof}

\begin{lemma}\label{lem:vt}
 There is $\newgc{vt}>0$ such that 
 \[
  \int_0^t\io \vet^2\le \gc{vt}(1+t)
 \]
 for every $t>0$ and $ε>0$.
\end{lemma}
\begin{proof}
 By \eqref{approximatesystem:v}, for every $t>0$ and $ε>0$ we have 
 \[
  \int_0^t\io \vet^2\le 3\int_0^t\io|Δ\ve|^2 + 3\int_0^t\io\ve^2 + 3\int_0^t\io\ue^2 \le \gc{vt}(1+t)
 \]
 if we rely on Lemmata~\ref{lem:vderivativebounds}, \ref{lem:vl2bound} and \ref{lem:u2xt} and set $\gc{vt}:=3(\gc{deltav2xt}+\gc{v2}+\gc{u2xt})$.
\end{proof}

\begin{lemma}\label{lem:nav4xt}
 There is $\newgc{nav4xt}>0$ such that for every $ε>0$ and every $t\in(0,∞)$, 
 \begin{equation}\label{bound-navfour-xt}
  \int_0^t\io |∇\ve|^4 \le \gc{nav4xt}(1+t)^3.
 \end{equation}
\end{lemma}
\begin{proof}
 We take $\newlc{gni1}>0$ and $\newlc{gni2}>0$ from a version of the Gagliardo--Nirenberg inequality (\cite[p.~126]{nirenberg} combined with \cite[Thm. I.19.1]{friedman} 
 ) stating that 
 \[
  \norm[\Lom4]{∇φ}^4 \le \lc{gni1} \norm[\Lom2]{Δφ}^2\norm[\Lom2]{∇φ}^2 + \lc{gni2} \norm[\Lom2]{∇φ}^4 
 \]
for all $φ\in \Wom22$ with $∂_{ν}φ = 0$ on $∂\Om$. Then, due to Lemma~\ref{lem:vderivativebounds},  
\begin{align*}
 \int_0^t\io |∇\ve|^4 &\le \lc{gni1}\int_0^t \norm[\Lom2]{Δ\ve}^2\norm[\Lom2]{∇\ve}^2 + \lc{gni2}\int_0^t \norm[\Lom2]{∇\ve}^4\\
 &\le \lc{gni1}\gc{nav2}(1+t)\int_0^t \norm[\Lom2]{Δ\ve}^2 + \lc{gni2}\gc{nav2}^2\int_0^t (1+t)^2\\
 &\le \lc{gni1}\gc{nav2}\gc{deltav2xt}(1+t)^2 + \lc{gni2}\gc{nav2}^2 t(1+t)^2
\end{align*}
holds for every $ε>0$ and $t\in(0,∞)$, and \eqref{bound-navfour-xt} follows with $\gc{nav4xt}:=\lc{gni1}\gc{nav2}\gc{deltav2xt}+\lc{gni2}\gc{nav2}^2$.
\end{proof}

\section{Initial regularity}\label{sec:initreg}
As announced in the introduction, the source for all regularity statements at the initial time will be the study of $\ddt\io Φ(\ue)$ for a well-chosen function $Φ$. Of $Φ$ we will require the following: 
\begin{itemize}
 \item finiteness of $\io Φ(u_0)$ (or, more precisely: boundedness of $\io Φ(\une)$), so that there is a chance for $ε$-independent bounds on $\sup_t\io Φ(\ue(\cdot,t))$,
 \item superlinear growth, in order to conclude some weak $\Lom1$ compactness; the condition can be stated as $\f{1}xΦ(x)\to \infty$ or, for differentiable functions, as $Φ'(x)\to \infty$ as $x\to \infty$,
 \item differentiability -- ideally $Φ\in C^2$ -- enabling any reasoning relying on derivatives, along with nonnegativity of $Φ$, $Φ'$ and $Φ''$,
 \item $Φ''(x)\le \f1x$, so as to facilitate estimates dealing with the cross-diffusive term during the computation of $\ddt\io Φ(\ue)$. 
\end{itemize}
The first two of these are a product of de la Vallée Poussin's theorem, also the third is often stated as part of it. Finding a proof in the literature which covers the differentiability seems  more difficult, however. As our subsequent reasoning crucially relies on it, we will recall a proof in the short Subsection~\ref{sec:poussin}. The key idea for the last of our  requirements is to diminish $Φ$ by setting $Φ_{new}''(x) = \min\set{\f1x,Φ_{old}''(x)}$. This would ensure $Φ_{new}\le Φ_{old}$ and thus preserve the first requirement; it could, however, violate the superlinear growth condition: Even if $f_1\notin L^1(ℝ_+)$ and $f_2\notin L^1(ℝ_+)$, it may happen that their pointwise minimum $\min\set{f_1,f_2}$ lies in $L^1(ℝ_+)$. We will therefore, roughly, try to set $Φ_{new}''(x)=\f1{x}$ whenever $Φ_{old}''(x)\ge\f1{x}$ but in case of $Φ_{old}''(x)<\f1{x}$ will also use $Φ_{new}''(x)=\f1{x}$, until $\int_0^x Φ_{new}''(y)dy$ has ``caught up'' with $\int_0^x Φ_{old}''(y)dy$, and only in case of $Φ_{old}''(x)<\f1{x}$ \textit{and} $Φ_{new}'(x)=Φ_{old}'(x)$ set $Φ_{new}''(x)=Φ_{old}''(x)$. Technically, we use a smoothed version of this idea and obtain the modified function $Φ$ as solution of a fixed point problem to be dealt with in Subsection~\ref{sec:locallyconvexstuff}. 

These preparations completed, in Lemma~\ref{lem:Phi-init} we will find a suitable $Φ$ and use it to ensure equiintegrability of $\set{\ue(\cdot,t)\mid ε>0, t\in[0,1)}$ in Lemma~\ref{lem:Phi-time}. 

Identification of possible limits of $u(\cdot,t)$ as $t\searrow 0$ will be the task of Subsection~\ref{sec:identify}, which will once more rely on the outcome of Lemma~\ref{lem:Phi-time}.

\subsection{De la Vallée Poussin's theorem: The source of the existence of \texorpdfstring{$Φ$}{Phi}}\label{sec:poussin}
The following proof is very close to the original by de la Vallée Poussin \cite[Remarque 23]{vallee-poussin}, although a small modification is necessary to account for smoothness (and convexity).
\begin{lemma}\label{lem:vallee-poussin}
 Let $\Om\subset ℝ^n$, $n\inℕ$, be a measurable set with finite measure and let $\set{u_{α}}_{α}$ be a family of nonnegative functions $u_{α}\colon \Om\to [0,∞]$ that is equiintegrable (also 'uniformly integrable') in the sense that 
\begin{equation}\label{unif-int}
 \forall ε>0 \; \exists δ>0\; \forall α\; \forall M\subset \Om: M \text{measurable with } |M|<δ \implies \int_M u_{α}<ε.
\end{equation}
Then there is a smooth, monotone, convex function $Φ\colon [0,∞)\to [0,∞)$ such that $Φ(0)=0$, $\lim_{x\to\infty}\f{Φ(x)}{x}=\infty$ and $\sup_{α} \io Φ(u_{α}) <\infty$.
\end{lemma}
\begin{proof}
 Firstly, \eqref{unif-int} implies the existence of some $\newlc{l1}>0$ such that $\io u_{α}\le \lc{l1}$ for all $α$.
 We let $(ε_k)_{k\inℕ}\subset (0,1)$ be a decreasing sequence such that $\sum_{k=1}^\infty ε_k <\infty$ and set $ε_{-1}:=ε_0:=1$. We furthermore introduce an increasing sequence $(M_k)_{k\inℕ}$ such that for every $k\inℕ$
\[
 \int_{\set{u_{α}\ge M_k}} u_{α} < ε_k^2 \qquad \text{for all } α
\]
and let $M_0:=0$. (Such a sequence exists by \eqref{unif-int} and because $|\set{u_{α}\ge M_k}|\le \f1{M_k} \sup_{α} \io u_α\le \f{c_1}{M_k}$ for every $α$ and for every $k\in ℕ$.) \\
With a function $ζ\in C^{∞}([0,1])$ such that $ζ^{(k)}(0)=ζ^{(k)}(1)=0$ for all $k\in ℕ$, $ζ(0)=0$, $ζ(1)=1$ and $ζ'\ge 0$ in $(0,1)$, we define a function $ϕ\colon[0,∞)\to [0,∞)$ by setting 
\[
 ϕ(x) := \kl{1-ζ\kl{\f{M_{k+1}-x}{M_{k+1}-M_k}}}\f1{ε_k} + ζ\kl{\f{M_{k+1}-x}{M_{k+1}-M_k}}\f1{ε_{k-1}}
\]
for $x\in[M_k,M_{k+1}]$, $k\inℕ_0$, and note that $$ϕ'(x)=\f1{M_{k+1}-M_k}\kl{\f1{ε_k}-\f1{ε_{k-1}}}ζ'\kl{\f{M_{k+1}-x}{M_{k+1}-M_k}} \ge 0$$ and that $ϕ\in C^{∞}([0,∞))$. Introducing $Φ(x):=\int_0^xϕ(y)dy$, $x\in[0,∞)$, we observe that $Φ$ is smooth, monotone and convex and $Φ(x)\le xϕ(x)$ for all $x>0$. Therefore, for every $α$, 
\begin{align*}
 \io Φ(u_{α}) &= \sum_{k=0}^\infty \int_{\set{M_k\le u_{α}\le M_{k+1}}} Φ(u_{α}) 
 \le \sum_{k=0}^\infty \int_{\set{M_k\le u_{α}\le M_{k+1}}} u_{α}ϕ(u_{α}) \\
&= \io u_{α} + \sum_{k\in ℕ}\f1{ε_k}\int_{\set{M_k\le u_{α}}} u_{α} \le \lc{l1} + \sum_{k\inℕ} ε_k <\infty. 
\end{align*}
Finally, 
\[
 \lim_{x\to\infty}\f{Φ(x)}x = \lim_{x\to\infty} ϕ(x) = \lim_{k\to\infty} \f1{ε_k} = \infty. \qedhere
\]
\end{proof}

\subsection{Adjusting the second derivative of \texorpdfstring{$Φ$}{Phi}}
\label{sec:locallyconvexstuff}
As we wish to obtain a modified variant of $Φ$ complying with all of the conditions mentioned at the beginning of Section~\ref{sec:initreg} from an application of Schauder's fixed point theorem, let us isolate a technical prerequisite in the following lemma. 
\begin{lemma}\label{lem:set-cpt}
 Let $f\in \Loneloc$ be nonnegative. The set 
\begin{equation}\label{def:X}
 X_f := \set{h\in\Loneloc\mid 0\le h\le f \quad a.e.}
\end{equation}
is a convex subset of $\Loneloc$ which is compact in the weak topology of each of the spaces $L^1((0,n))$, $n\inℕ$.
\end{lemma}
\begin{proof}
 By the Eberlein--Shmulyan theorem, \cite[Thm.~V.6.1]{dunford-schwartz}, it is sufficient to check for sequential weak compactness. 
Let $(h_k)_{k\in ℕ}$ be a sequence in $X_f$. Given $n\inℕ$ we observe that $\set{\norm[L^1((0,n))]{h_k}}_{k\in ℕ}$ is bounded (by $\norm[L^1((0,n))]{f}$) and that 
\[
 \forall ε>0\; \exists δ>0\; \forall M\subset(0,n):\; M\text{ measurable with } |M|<δ \implies \int_M f<ε,
\]
which due to $0\le h_k\le f$ for all $k\in ℕ$ implies 
\[
 \forall ε>0\; \exists δ>0\; \forall k\in ℕ\;\forall M\subset(0,n):\; M\text{ measurable with } |M|<δ \implies \int_M h_k<ε.
\]
The sequence $\set{h_k}_{k\inℕ}$ therefore is equiintegrable on $(0,n)$ and by the Dunford--Pettis theorem (\cite[Thm. IV.8.9]{dunford-schwartz}) there is a subsequence 
which converges weakly in $L^1((0,n))$ to some $h\in L^1((0,n))$
, where $h\le f$ a.e. on $(0,n)$. 
Convexity of the set is obvious. 
\end{proof}

We will, in Lemma~\ref{lem:Phi-init}, obtain $Φ$ (or rather $Φ''$) from an application of the following lemma to $f=Φ''$ with $Φ$  from Lemma~\ref{lem:vallee-poussin} and $g(x)=\f1x$.
\begin{lemma}\label{lem:adjustPhi}
 Let $f\in C([0,∞))$ and $g\in C((0,∞))$ be nonnegative functions and such that $\lim_{x\to 0^+} g(x)\in [0,∞]$ exists and $f\not\in L^1((0,∞))$ as well as $g\not\in L^1((1,∞))$. Then there is $h\in C([0,∞))$ such that 
\begin{equation}\label{h:nice}
0\le h(x)\le g(x)\qquad \text{ and }\qquad \int_0^xh(y)dy\le \int_0^x f(y)dy
\end{equation}
for all $x>0$ and $h\notin L^1((0,∞))$. 
\end{lemma}
\begin{proof}
 Without loss of generality, we may assume that $g\in C([0,∞))$. (For if $g\ge f$, the choice $h:=f$ suffices to conclude, and if otherwise $g\notin\Loneloc$, then $x_0:=\inf\set{x>0\mid f(x)=g(x)}$ is positive and redefining $g(x):=f(x)$ for $x\in(0,x_0)$ will ensure $g\in C([0,∞))$ and will not affect the construction in \eqref{def:h} below.) \\
We pick a cut-off function $φ\in C^{∞}(ℝ)$ such that $φ\equiv 0$ on $[-\f12,∞)$, $φ\equiv 1$ on $(-∞,-1]$ and $φ'\le 0$ in $ℝ$ and note that $φ$ is Lipschitz continuous with a Lipschitz constant that we will denote by $L_{φ}$. 
For $h\in \Loneloc$ we then define $Φ(h)$ by 
\begin{align}\label{def:h}
 &Φ(h)(x):=
 &\begin{cases}
           g(x),& \text{if } f(x)\ge g(x),\\
	    φ\kl{\int_0^x(h(y)-f(y))dy} g(x) + \kl{1-φ\kl{\int_0^x(h(y)-f(y))dy}}f(x), &\text{if } f(x)<g(x)
          \end{cases}
\end{align}
for $x\in(0,∞)$. Apparently, $Φ(h)\in C([0,∞))$. 
Moreover, for $h_1,h_2\in\Loneloc$ and $x\in(0,∞)$, we have 
\begin{align}\nn
 &|(Φ(h_1)-Φ(h_2))(x)|=\begin{cases}
                       0& \text{if } f(x)\ge g(x),\\
			\lvert g(x)-f(x)\rvert\kl{φ(\int_0^x h_1-f)-φ(\int_0^x h_2-f)}&\text{if } f(x)<g(x)
                      \end{cases}\\
&\qquad \le \lvert g(x)-f(x)\rvert L_{φ}\left\lvert\int_0^x (h_1-h_2)\right\rvert \le \lvert g(x)-f(x) \rvert L_{φ}\left\lvert\int_0^x (h_1-h_2)\right\rvert.\label{forcontinuity}
\end{align}
We show that $Φ$ is sequentially continuous as map from $\Loneloc$ with weak topology to $\Loneloc$ with its usual topology. Let $(h_{1,k})_{k\inℕ} \subset \Loneloc$ be a sequence which weakly converges to $h_2\in \Loneloc$. Let $n\inℕ$. Then, for every $x\in (0,n)$, $\int_0^x (h_{1,k}-h_2) \to 0$ as $k\to \infty$. Moreover, there is $\newlc{l1bd}>0$ such that $\left\lvert\int_0^x (h_{1,k}-h_2)\right\rvert\le \norm[L^1((0,n))]{h_{1,k}}+\norm[L^1((0,n))]{h_2}\le \lc{l1bd}\in L^1((0,n))$ for every $x\in(0,n)$ and $k\inℕ$. 
An integration of \eqref{forcontinuity} and Lebesgue's theorem therefore show that 
\[
 \norm[L^1((0,n))]{Φ(h_{1,k})-Φ(h_2)} \le \normm{C([0,n])}{f-g} L_{φ}\int_0^n\left\lvert\int_0^{x} (h_{1,k}-h_2)\right\rvert\mathrm{d}x\to 0
\]
as $k\to \infty$ and thus the desired sequential continuity. 
Moreover, by definition \eqref{def:h}, $0\le Φ(h)(x)\le \max\set{f,g}$ for every $h\in \Loneloc$, or, in the notation from \eqref{def:X}, $Φ(h)\in X_{\max\set{f,g}}$. 
In order to use the compactness statement from Lemma~\ref{lem:set-cpt}, let us introduce $Φ_1\colon X_{\max\set{f,g}}\to X_{\max\set{f,g}}$, $Φ_1(h)(x)=Φ(h)(x)$, $x\in(0,1)$, $h\in X_{\max\set{f,g}}$, where we consider $X_{\max\set{f,g}}$ as a subset of the space $L^1((0,1))$ with its weak topology. 

The set $X_{\max\set{f,g}}$ is convex and compact (cf. Lemma~\ref{lem:set-cpt}) and $Φ_1$ maps $X_{\max\set{f,g}}$ into itself and is continuous (for the step from sequential continuity as shown above to continuity we rely on weak compactness of the set and the resulting metrizability of its weak topology, cf. \cite[Thm. V.6.3]{dunford-schwartz}). Schauder's theorem (in the form of, e.g., \cite[Thm. V.10.5]{dunford-schwartz}) ensures the existence of a fixed point $h_1\in L^1((0,1))$. If $h_n\in L^1((0,n))$ has been constructed for some $n\inℕ$, we let $Φ_{n+1}\colon X_{\max\set{f,g}}\subset L^1((0,n+1)) \to X_{\max\set{f,g}}$, 
\[
 Φ_{n+1}(h)(x) = \begin{cases}
                  h_n(x),& x\in(0,n),\\
                  Φ(h)(x),& x\in[n,n+1),
                 \end{cases}
\]
for $h\in L^1((0,n+1))$, and by the same reasoning as before obtain a fixed point $h_{n+1}$ of $Φ_{n+1}$. We note that $h_{n_2}|_{(0,n_1)}=h_{n_1}$ for every $n_2\ge n_1$ and that $Φ(h_n)=h_n$ on $(0,n)$ for $n\inℕ$. Defining $h(x):=h_n(x)$ for $x\in(n-1,n)$, we obtain some $h\in\Loneloc$ such that $h=Φ(h)\in C([0,∞))$.
The construction in \eqref{def:h} ensures that $h\le g$. In order to see the second part of \eqref{h:nice}, we define $ψ(x):=\int_0^x(h(y)-f(y))dy$ and observe that due to continuity of $h$ and $f$, $ψ\in C^1([0,∞))$. We let $x_0:=\inf\set{x\in[0,∞)\mid ψ(x)>0}$ (assuming that this set were non-empty), so that $ψ(x_0)=0$. Moreover, by this definition as an infimum and by continuity of $ψ$, we then could find $x_1>x_0$ such that $ψ(x_1)>0$ and that $ψ(x)>-\f12$ (and hence $φ(ψ(x))=0$) for all $x\in(x_0,x_1)$. For any $x\in(x_0,x_1)$ we would then conclude from \eqref{def:h} that $h(x)=g(x)\le f(x)$ or $h(x)=φ(ψ(x))g(x)+(1-φ(ψ(x)))f(x)=f(x)$, which would lead to the contradiction $0<ψ(x_1)=ψ(x_0) + \int_{x_0}^{x_1} (h(y)-f(y))dy\le ψ(x_0) + 0=0$.\\
Finally, in order to show $h\notin L^1((0,∞))$, we assume that $h\in L^1((0,∞))$ and from nonnegativity of $h$ could infer the existence of $\newlc{hl1}>0$ such that $\int_0^x h(y)dy\le \lc{hl1}$ for all $x>0$. As $\int_0^x f(y)dy\to \infty$ as $x\to \infty$ (due to $0\le f\notin L^1((0,∞))$), there were $x_2\in [0,∞)$ such that $\int_0^x f(y)dy \ge\lc{hl1}+1 > \int_0^x h(y) dy +1 $ for all $x>x_2$. Therefore, for every $x>x_2$, $φ(\int_0^x h_1-f) = 1$, and hence $h(x)=g(x)$ according to \eqref{def:h}. Thus, $\int_0^{∞} h(y)dy \ge \int_{x_2}^{∞} h(y)dy = \int_{x_2}^{∞} g(y)dy = ∞$, contradicting the assumption $h\in L^1((0,∞))$.
\end{proof}

\subsection{Equiintegrability}
\begin{lemma}\label{lem:Phi-init}
 There is a function $Φ\colon [0,∞)\to [0,∞)$ such that 
 \begin{align}
  &Φ \text{ is twice differentiable, monotone and convex} \label{Phi:regularity}\\
  &\f{Φ(x)}{x} \to \infty \quad \text{ and } Φ'(x) \to \infty \qquad \text{as } x\to \infty\label{Phi:xtoinfty}\\
  &Φ''(x)\le \f1x \qquad \text{for all } x>0\label{Phi:Phi2leq}
 \end{align}
 and such that with some $\newgc{Phi-init}>0$ 
 \begin{equation}
  \io Φ(\une)\le \gc{Phi-init} \label{eq:Phiune}
 \end{equation}
 holds for every $ε>0$.
\end{lemma}
\begin{proof}
 Due to \eqref{init-eps:conv:u}, $\set{\une\midε>0}$ is equiintegrable. By de la Vallée Poussin's lemma (Lemma~\ref{lem:vallee-poussin}), there is a function $Ψ\colon [0,∞)→[0,∞)$ satisfying \eqref{Phi:regularity}, \eqref{Phi:xtoinfty} and \eqref{eq:Phiune} (with $Φ$ replaced by $Ψ$). Noting that $Ψ''\notin L^1((0,∞))$, we apply Lemma~\ref{lem:adjustPhi} with $f=Ψ''$ and $g(x)=\f1x$, obtaining some function $h$ as described there. We then define $Φ(x)=\int_0^x\kl{Ψ'(0)+\int_0^y h(z) \dz}\dy$ for $x\ge 0$. Then \eqref{Phi:regularity} is satisfied; due to $h\notin L^1((0,∞))$, \eqref{Phi:xtoinfty} holds true; the first part of \eqref{h:nice} ensures \eqref{Phi:Phi2leq}, whereas the second shows that $Φ'(x)\le Ψ'(x)$ and consequently $Φ(x)\le Ψ(x)$ for all $x\ge 0$ and thus \eqref{eq:Phiune}.
\end{proof}

\begin{lemma}\label{lem:Phi-time}
With $Φ$ as introduced in Lemma~\ref{lem:Phi-init}, for every $T>0$ there is $\newgc{Phi}=\gc{Phi}(T)>0$ satisfying 
 \begin{align*}
  \io Φ(\ue\att) \le \gc{Phi} \quad \text{and } \quad \int_0^T\io Φ'(\ue)\ue^2\le \gc{Phi}
 \end{align*}
 for every $ε>0$ and $t\in[0,T)$. 
\end{lemma}
\begin{proof}
 Integration by parts and Young's inequality let us conclude from \eqref{approximatesystem:u} that 
\begin{align}
 \ddt \io Φ(\ue) 
 &= -\io Φ''(\ue)|∇\ue|^2 + χ\io Φ''(\ue)\ue∇\ve\cdot ∇\ue + κ\io Φ'(\ue)\ue - μ\io Φ'(\ue)\ue^2\nn\\
 &\le χ^2\io Φ''(\ue)\ue^2  |∇\ve|^2 + κ\io Φ'(\ue)\ue - μ\io Φ'(\ue)\ue^2\nn\\
 &\le \io (Φ''(\ue))^2\ue^4 + χ^4\io |∇\ve|^4 + κ\io Φ'(\ue)\ue - μ\io Φ'(\ue)\ue^2\label{ddtPhi}
\end{align}
holds on $(0,T)$.
We choose $c_{κ,μ}>0$ such that 
\[
 (κx-μx^2)Φ'(x) \le c_{κ,μ} - \f{μ}2x^2Φ'(x) \qquad \text{for all } x\ge 0
\]
and, relying on \eqref{Phi:xtoinfty}, $c_{Φ,μ}>0$ such that 
\[
 \kl{1-\f{μ}4Φ'(x)}x^2 \le c_{Φ,μ} \qquad \text{for all } x\ge 0,
\]
and use these definitions together with \eqref{Phi:Phi2leq} to turn \eqref{ddtPhi} into 
\begin{align*}
 \ddt \io Φ(\ue)  &\le \io \ue^2 + χ^4\io |∇\ve|^4 + \io c_{κ,μ} - \f{μ}2\io Φ'(\ue)\ue^2\\
 &\le (c_{κ,μ}+c_{Φ,μ})|\Om| - \f{μ}4 \io Φ'(\ue)\ue^2 + χ^4\io |∇\ve|^4\qquad \text{on } (0,T).
\end{align*}
Integration with respect to time results in the claim, due to Lemma~\ref{lem:nav4xt} and the bound \eqref{eq:Phiune} on the value at time $0$ from Lemma~\ref{lem:Phi-init}.
\end{proof}

\subsection{Identification of the limit as \texorpdfstring{$t\searrow 0$}{t to 0}. Estimates for \texorpdfstring{$\uet$}{the time derivative of uε}}\label{sec:identify}
The equiintegrability statement contained in Lemma~\ref{lem:Phi-time} (or resulting from this lemma in the limit $ε\to 0$, cf. proof of Lemma~\ref{lem:ucts}) is sufficient to conclude that for every sequence $(t_k)_{k\inℕ}$ with $t_k\to0$ there is a subsequence $(t_{k_j})_{j\inℕ}$ such that $u(\cdot,t_{k_j})$ converges in $\Lom1$. This does not capture the value of $\lim_{j}u(\cdot,t_{k_j})$ and would still admit different accumulation points of $(u(\cdot,t_k))_{k\inℕ}$. To make sure that this limit coincides with $u_0$, we turn our attention to estimates for $\uet$.

\begin{lemma}\label{lem:normuet}
 There is $\newgc{ut}>0$ such that for every $ε>0$ and $t\in(0,∞)$, 
\[
  \norm[(W_0^{3,3}(\Om))^*]{\uet\att} \le \gc{ut} (1+t) + \gc{ut} \norm[\Lom2]{\ue\att}^2.
\]
\end{lemma}
\begin{proof}
For any $ε>0$ and $t\in (0,∞)$, the norm of $\uet\att$ in $(W_0^{3,3})^*$ 
can be computed as  
\[
 \norm[(W_0^{3,3}(\Om))^*]{\uet\att} = \sup \lset{\left\lvert\io \uet φ\right\rvert; 
 \; φ\in C_c^{∞}(\Om), \norm[W_0^{3,3}(\Om)]{φ}\le 1}
\]
Due to the embedding $W_0^{3,3}(\Om)\hookrightarrow W^{2,∞}(\Om)$, there is $c_1>0$ such that for any such $φ$, $\norm[\Wom2{∞}]{φ}\le c_1$, and accordingly 
\begin{align*}
 \left\lvert \io \uet φ\right\rvert &= \bigg| \io Δ\ueφ - χ\io ∇\cdot(\ue∇\ve)φ + κ\io \ueφ - μ\io \ue^2φ\bigg|\\
 &= \bigg| \io \ueΔφ - χ\io \ue∇\ve\cdot∇φ + κ\io \ueφ - μ\io \ue^2φ\bigg|\\
 &\le c_1\norm[\Lom1]{\ue} + c_1 |χ|\norm[\Lom1]{\ue∇\ve} + c_1|κ|\norm[\Lom1]{\ue} + c_1μ\norm[\Lom2]{\ue}^2\\
 &\le c_1(1+|κ|) \norm[\Lom1]{\ue} + c_1|χ|\norm[\Lom2]{∇\ve}^2 + c_1(|χ|+μ)\norm[\Lom2]{\ue}^2\\
 &\le c_1(1+|κ|) \gc{mass} + c_1|χ| \gc{nav2}(1+t) + c_1(|χ|+μ)\norm[\Lom2]{\ue}^2 \quad \text{for all } ε>0, t\in(0,∞),
\end{align*}
where we have successively employed integration by parts, Hölder's and Young's inequality and Lemmata \ref{lem:mass} and \ref{lem:vderivativebounds}. 
\end{proof}
For controlling the last summand in the estimate of Lemma~\ref{lem:normuet}, we rely on Lemma~\ref{lem:Phi-time}.
\begin{lemma}\label{lem:usquare-equiint}
 For every $T>0$, the family 
 \[
  \set{[0,T]\ni t\mapsto \norm[\Lom2]{\ue\att}^2 | ε>0}
 \]
 of functions on $[0,T]$ is equiintegrable.
\end{lemma}
\begin{proof}
 Introducing 
 \[
  F(x):= xΦ'\kl{\sqrt{\f x{|\Om|}}}, \qquad x>0,
 \]
 with $Φ$ from Lemma~\ref{lem:Phi-init}, 
 we know that 
 \[
  F'(x)=Φ'\kl{\sqrt{\nf x{|\Om|}}} + \f12 \sqrt{\nf x{|\Om|}}\,  Φ''\kl{\sqrt{\nf x{|\Om|}}} \ge Φ'\kl{\sqrt{\nf x{|\Om|}}} \quad \text{for all } x>0, 
 \]
 and hence $F'(x)\to \infty$ as $x\to \infty$. In other words, for every $l>0$ there is $x_l>0$ such that $F'>l$ on $(x_l,∞)$. 
 We define 
 \[
  Ψ:=\conv F
 \]
 to be the convex hull of $F$ (cf. e.g. \cite[5.16]{tiel1984convex})
 and observe that $Ψ$ lies above each of the convex functions $x\mapsto l(x-x_l)_+$, so that $\f{Ψ(x)}{x}\to \infty$ as $x\to \infty$.
Its convexity allows us to invoke Jensen's inequality and estimate 
\begin{align*}
 \int_0^T Ψ(\norm[\Lom2]{\ue}^2) dt &= \int_0^T Ψ\kl{\io|\Om|\ue^2\f{dx}{|\Om|}} dt \le 
  \int_0^T \io Ψ(|\Om|\ue^2) \f{dx}{|\Om|}\\
  &\le \int_0^T\io F\kl{|\Om|\ue^2}\f{dx}{|\Om|}
  =\int_0^T\io |\Om|\ue^2 Φ'(\ue)\f{dx}{|\Om|}\\
  &=\int_0^T\io \ue^2 Φ'(\ue)
\end{align*}
for every $ε>0$. 
The final integral is bounded by Lemma~\ref{lem:Phi-time}, and the resulting boundedness of the integral of a superlinear function of $\norm[\Lom2]{\ue}^2$ ensures the latter's equiintegrability, as claimed.
\end{proof}

\begin{lemma}\label{lem:uetequiint}
 For every $T>0$, the family 
 \[
  \set{[0,T]\ni t\mapsto \norm[\Womstar33]{\uet \att} | ε>0}
 \]
 of functions on $[0,T]$ is equiintegrable.
\end{lemma}
\begin{proof}
In light of Lemma~\ref{lem:normuet}, this is a simple consequence of Lemma~\ref{lem:usquare-equiint}. 
\end{proof}
This equiintegrability statement can be turned into the following assertion on $\lim_{(ε,t)\to 0} \ue(\cdot,t)$:
\begin{lemma}\label{lem:ue-at0}
 For every $η>0$ there are $t_0>0$ and $ε_0>0$ such that for every $t\in(0,t_0)$ and every $ε\in(0,ε_0)$, 
 \[
  \norm[\Womstar33]{\ue\att - u_0}\le η.
 \]
\end{lemma}
\begin{proof}
 Let $η>0$. In accordance with \eqref{init-eps:conv:u} and aided by the embedding $\Lom1\hookrightarrow\Womstar33$, we can find $ε_0>0$ such that 
 \[
  \text{for all } ε\in(0,ε_0): \qquad \norm[\Womstar33]{\une-u_0}\le \f{η}2.
 \]
 Relying on the outcome of Lemma~\ref{lem:uetequiint}, we moreover choose $t_0>0$ such that for all sets $M\subset[0,1]$ with $|M|\le t_0$ and all $ε>0$, 
 \[
  \int_M \norm[\Womstar33]{\uet}\le \f{η}2.
 \]
 For every $t\in(0,t_0)$ and every $ε\in(0,ε_0)$, we then have 
 \begin{align*}
  \norm[\Womstar33]{\ue\att-u_0}&\le \norm[\Womstar33]{\une-u_0}+\norm[\Womstar33]{\ue\att-\une} \\
  &\le \f{η}2+\int_{(0,t)}\norm[\Womstar33]{\uet} \le η.\qedhere
 \end{align*}
\end{proof}

\section{Regularity outside \texorpdfstring{$t=0$}{t=0}. Time-local bounds}
Relying on the spatio-temporal bounds on $|∇\ve|^4$ from Lemma~\ref{lem:nav4xt}, we do not base the first stepping stone of our estimates on the functional $\f12\io \ue^2+\f14\io|∇\ve|^4$ (cf. \cite[Lemma~3.3]{BBTW}), but on $\f12\io \ue^2$ alone. 
The approach adheres to a classical way to obtain estimates in parabolic PDEs (cf. \cite[p.140]{LSU}), but is less used in a chemotaxis context -- the reason often lying in lack of a priori information alike Lemma~\ref{lem:nav4xt}. It avoids having to deal with boundary integrals involving $∂_{ν}|∇\ve|^2$ and thus may be of interest in some related models. Let us note that it would, instead, have been possible to adapt the reasoning of \cite[Lemma~3.3]{BBTW} (resulting in a proof of similar length and simplicity). \\

On the space 
\[
 V:=L^2_{loc}((0,∞);\Wom12)\cap C((0,∞);\Lom2)
\]
we introduce the seminorms 
\[
 \norm[V,t_1,t_2]{φ} = 
 \sup_{t\in(t_1,t_2)}\norm[\Lom2]{φ(t)}+\norm[L^2(\Om\times(t_1,t_2))]{∇φ}, \qquad t_2>t_1>0
\]
and can derive the following inequality (cf. \cite[p.75]{LSU}):
\begin{lemma}\label{lem:estimateVnorm}
There is $\newgc{Vgni}>0$ such that for every $t_1>0$ and $t_2>t_1$ 
\[
 \norm[L^4(\Om\times(t_1,t_2))]{φ}^4\le \gc{Vgni}(1+t_2)\norm[V,t_1,t_2]{φ}^4
\]
holds for every $φ\in V$. 
\end{lemma}
\begin{proof}
 We take $\newlc{gni1}>0$ and $\newlc{gni2}>0$ from the Gagliardo-Nirenberg inequality 
 \[
  \norm[\Lom4]{φ}^4\le \lc{gni1}\norm[\Lom2]{∇φ}^2\norm[\Lom2]{φ}^2 + \lc{gni2}\norm[\Lom2]{φ}^4 \qquad \text{for all }φ\in \Wom12.
 \]
Then for every $φ\in V$, $t_2>t_1>0$, 
\begin{align*}
 \norm[L^4(\Om\times(t_1,t_2))]{φ}^4 
 &\le \lc{gni1} \int_{t_1}^{t_2} \norm[\Lom2]{∇φ}^2\norm[\Lom2]{φ}^2 + \lc{gni2}\int_{t_1}^{t_2} \norm[\Lom2]{φ}^4\\
 &\le \lc{gni1} \norm[V,t_1,t_2]{φ}^2\int_{t_1}^{t_2} \norm[\Lom2]{∇φ}^2 + \lc{gni2} (t_2-t_1) \norm[V,t_1,t_2]{φ}^4\\
 &\le (\lc{gni1}+\lc{gni2}(t_2-t_1))\norm[V,t_1,t_2]{φ}^4.\qedhere
\end{align*}
\end{proof}

\begin{lemma}\label{lem:uV}
 For every $T>0$ and $τ\in(0,T)$ there is $\newgc{uV}(τ,T)>0$ such that for every $ε>0$, 
 \[
  \norm[V,τ,T]{\ue} \le \gc{uV}(τ,T).
 \]
\end{lemma}
\begin{proof}
 We fix $T>0$ and $τ\in(0,T)$. 
 For any $ε>0$ testing \eqref{approximatesystem:u} by $\ue$, we see that with 
 $\newlc{logistic}=\f1{|\Om|}\sup_{x\ge 0} (κx^2-\f{μ}2x^3)$ and after an application of Young's inequality, 
 \begin{align*}
  \f12\ddt\io \ue^2 \le -\f12\io |∇\ue|^2 + \f{χ^2}2\io \ue^2|∇\ve|^2 + \lc{logistic} - \f{μ}2\io \ue^3\quad\text{in } (0,T).
 \end{align*}
Integration with respect to time, Hölder's inequality and Lemma~\ref{lem:estimateVnorm} show that hence for every $ε>0$, $t_1>0$, $t_2\in(t_1,T]$ and $t\in(t_1,t_2)$, 
\begin{align*}
 \max&\setl{\io \ue^2\att, \int_{t_1}^t\io |∇\ue|^2} +μ\int_{t_1}^{t}\io \ue^3\le \io \ue^2(\cdot,t_1) + χ^2\int_{t_1}^t\io \ue^2|∇\ve|^2+2\lc{logistic}(t-t_1)\\
 &\le \io \ue^2(\cdot,t_1)+ χ^2\kl{\int_{t_1}^{t_2} \io \ue^4}^{\f12}\kl{\int_{t_1}^{t_2}\io |∇\ve|^4}^{\f12} + 2\lc{logistic}(t_2-t_1)\\
 &\le \io \ue^2(\cdot,t_1)+ χ^2\sqrt{\gc{Vgni}(1+t_2)} \norm[V,t_1,t_2]{\ue}^2 \kl{\int_{t_1}^{t_2}\io |∇\ve|^4}^{\f12} + 2\lc{logistic}t_2.
\end{align*}
We let $\newlc{coeffabs}=16\lc{logistic}T$ and $\newlc{coeffv}=4χ^2\sqrt{\gc{Vgni}(1+T)}$ and take the supremum over $t\in(t_1,t_2)$, ending up with 
\begin{equation}\label{Vestimate1}
 \norm[V,t_1,t_2]{\ue}^2 \le 2\norm[\Lom2]{\ue(\cdot,t_1)}^2+ \f{\lc{coeffabs}}4 + \f{\lc{coeffv}}2 \norm[L^4(\Om\times(t_1,t_2))]{∇\ve}^2\norm[V,t_1,t_2]{\ue}^2
\end{equation}
for every $ε>0$, $t_1>0$ and $t_2\in(t_1,T)$.
Aided by Lemma~\ref{lem:u2xt} we pick $t_0=t_0(ε)\in(\f{τ}2,τ)$ such that 
\[
 \io \ue^2(\cdot,t_0) \le \f{2}{τ} \gc{u2xt} (1+τ) =: \newlc{att0}. 
\]
Given $t_k=t_k(ε)\in(0,T]$, we choose $t_{k+1}=t_{k+1}(ε)\in(t_k,T)$ such that 
\begin{equation}\label{def:tk}
\int_{t_k}^{t_{k+1}}\io |∇\ve|^4 = \f1{\lc{coeffv}^2}
\end{equation}
if such $t_{k+1}$ exists, or $t_{k+1}=T$ otherwise. Then $\int_{t_0}^{t_k} \io |∇\ve|^4 = \f{k}{\lc{coeffv}^2}$ if $t_k<T$, which means that for $k>\newlc{k0}:=1+\lc{coeffv}^2\gc{nav4xt}(1+T)> \lc{coeffv}^2\int_0^T\io |∇\ve|^4$ (the last inequality herein draws upon Lemma~\ref{lem:nav4xt}) clearly $t_k=T$. Moreover, \eqref{Vestimate1} for $t_1=t_k$ and $t_2=t_{k+1}$ for any $ε>0$ turns into 
\begin{equation}\label{Vestimate2}
 \norm[V,t_k,t_{k+1}]{\ue}^2 \le 4\norm[\Lom2]{\ue(\cdot,t_k)}^2 + \f{\lc{coeffabs}}2\quad\text{for }k\inℕ.
\end{equation}
We now introduce $y_{ε,k}:=\norm[V,t_0,t_k]{\ue}^2$ for $k\inℕ$ (and $y_{ε,0}=\io\ue^2(\cdot,t_0)$) and see that by \eqref{Vestimate2}
\[
 y_{ε,k+1} \le \kl{\sqrt{y_{ε,k}}+ \norm[V,t_k,t_{k+1}]{\ue}}^2\le 2y_{ε,k}+2\norm[V,t_k,t_{k+1}]{\ue}^2 \le 2y_{ε,k} + 8y_{ε,k} +c_2,
\]
so that $y_{ε,k}\le 10^ky_{ε,0}+\sum_{l=0}^{k-1} 10^l\lc{coeffabs}$ and thus $y_{ε,k} \le 10^{\lc{k0}}\lc{att0} + \f{10^{\lc{k0}}-1}9\lc{coeffabs}=:\newlc{Vbound}$ for every $k\in ℕ$. In particular, for any $ε>0$, 
\[
 \norm[V,τ,T]{\ue}\le\sqrt{\lc{Vbound}}.\qedhere
\]
\end{proof}

\begin{lemma}\label{lem:navq}
 Let $q\in [1,∞)$. For every $T>0$ and $τ\in(0,T)$ there is $\newgc{navq}(q,τ,T)>0$ such that 
 \[
  \norm[\Lom q]{∇\ve\att} \le \gc{navq}(q,τ,T) \qquad\text{for every } ε>0 \text{ and } t\in(τ,T). 
 \]
\end{lemma}
\begin{proof}
 Without loss of generality we assume $q\ge 4$. 
 Because of Lemma~\ref{lem:nav4xt}, for every $ε>0$ we can find $t_0=t_0(ε)\in(\f{τ}3,\f{τ}2)$ such that 
 \[
  \io |∇\ve(\cdot,t_0(ε))|^4 \le \f{2\gc{nav4xt}(1+τ)^3}{τ} =: \newlc{att0}.
 \]
 With $\newlc{sg}>0$ being the constant in the semigroup estimate  \cite[Lemma~1.3]{win_aggregation} we then obtain for every $ε>0$ and every $t\in(τ,T)\subset(t_0(ε),T)$ that 
 \begin{align*}
  \norm[\Lom q]{∇\ve\att} & \le \norm[\Lom q]{∇e^{(t-t_0)(Δ-1)}\ve(\cdot,t_0)} + \lc{sg} \int_{t_0}^t \norm[\Lom q]{∇e^{(t-s)(Δ-1)}\ue\ats} \ds\\
  \le& \lc{sg}e^{-(t-t_0)}\kl{1+(t-t_0)^{-\f14+\f1q}} \norm[\Lom4]{∇\ve(t_0)} \\
  &+ \lc{sg} \int_{t_0}^t \kl{1+(t-s)^{-\f12-(\f12-\f1q)}}e^{-(t-s)} \norm[\Lom2]{\ue\ats}\ds\\
  \le& \sqrt[4]{\lc{att0}}\lc{sg} \kl{1+\kl{\nf{τ}2}^{\f1q-\f14}}+\lc{sg}\newlc{int}\gc{uV}\kl{\f{τ}3,T}
 \end{align*}
 holds for every $t\in(τ,T)$ and every $ε>0$, where we have set $\lc{int}=\int_0^{∞}(1+σ^{-1+\f1q})e^{-σ}\dsigma$ and relied on Lemma~\ref{lem:uV}.
\end{proof}

\begin{lemma}\label{lem:up}
 Let $p\inℕ$. For every $T>0$ and $τ\in(0,T)$ there are $\newgc{up}(p,τ,T)>0$ and $\newgc{upext}(p,τ,T)>0$ such that 
 \[
  \io \ue^p\att \le \gc{up}(p,τ,T) \quad \text{ and } \quad \int_{τ}^T\io \ue^{p+1}\le \gc{upext}(p,τ,T)
 \]
 hold for every $ε>0$ and every $t\in(τ,T)$.
\end{lemma}
\begin{proof}
 For $p=1$, this is true by Lemmata~\ref{lem:mass} and \ref{lem:u2xt}. Assuming that the claim holds for some $p-1\inℕ$ and given $T>0$ and $τ\in(0,T)$, for every $ε>0$ we can find $t_0=t_0(ε)\in(\f{τ}2,τ)$ such that $\io \ue^p(t_0)\le \f2{τ}\gc{upext}(p-1,\f{τ}2,τ)=:\newlc[,p]{att0}$. 
 Moreover, with $\newlc[,p]{logistic}:=\f1{|\Om|}\sup \set{κx^p-\f{μ}3x^{p+1}\mid x\ge 0}$ and by Young's inequality then 
 \begin{align*}
  \f1p\ddt \io \ue^p &= -(p-1)\io \ue^{p-2}|∇\ue|^2 + (p-1)χ\io \ue^{p-1} ∇\ue\cdot∇\ve + κ\io \ue^p - μ\io \ue^{p+1}\\
  &\le 
  \f{μ}3\io \ue^{p+1} + \kl{\f{3}{μ}}^p(p-1)^{p+1}χ^{2p+2}\io |∇\ve|^{2(p+1)} + \lc[,p]{logistic} - \f{2μ}3\io \ue^{p+1}\quad \text{in } (t_0,T), 
 \end{align*}
 so that with $\newlc[,p]{young}=p(\nf{3}{μ})^p(p+1)^{p+1}χ^{2p+2}$ we obtain 
 \begin{align*}
  \io\ue^p\att + \f{pμ}3\int_{τ}^t\io \ue^{p+1} 
  &\le \io \ue^p(\cdot,t_0) + p\lc[,p]{logistic} (T-\f{τ}2) + \lc[,p]{young}\int_{t_0}^t\io |∇\ve|^{2(p+1)}\\
  &\le \lc[,p]{att0} + p\lc[,p]{logistic}T + \lc[,p]{young} \kl{\gc{navq}(2p+2,\nf{τ}2,T)}^{2p+2}=:\gc{up}(p,τ,T)
 \end{align*}
for every $ε>0$ and every $t\in(τ,T]$, and finally set $\gc{upext}(p,τ,T):=\f3{pμ}\gc{up}(p,τ,T)$.
\end{proof}

\begin{lemma}\label{lem:uinfty}
For every $T>0$ and $τ\in(0,T)$ there is $\newgc{uinfty}(τ,T)>0$ such that for every $ε>0$ and every $t\in(τ,T)$,
\[
 \norm[\Lom{∞}]{\ue\att}\le \gc{uinfty}(τ,T).
\]
\end{lemma}
\begin{proof}
With $\newlc{sg}>0$ taken from semigroup estimates (cf. \cite[Lemma 1.3]{win_aggregation}) and $\newlc{logistic}:=\sup_{x\ge0} (κx-μx^2)$,
\begin{align*}
 \norm[\Lom\infty]{\ue\att} \le& \lc{sg}(1+(t-t_0)^{-1})\norm[\Lom1]{\ue(\cdot,t_0)} + \lc{logistic}(t-t_0)\\
 &+ \lc{sg}|χ|\int_{t_0}^t\kl{1+(t-s)^{-\f12-\f14}}\norm[\Lom8]{\ue(\cdot,s)}\norm[\Lom8]{∇\ve(\cdot,s)} \ds 
\end{align*}
holds for every $t_0>0$ and $t>t_0$, and we may conclude by Lemmata~\ref{lem:mass}, \ref{lem:navq} and \ref{lem:up}.
\end{proof}

\begin{lemma}\label{lem:ualpha}
 For every $T>0$ and $τ\in(0,T)$ there are $α>0$ and $\newgc{uholder}(τ,T)>0$ such that, for all $ε>0$,  
 \[
  \normm{C^{α,\f{α}2}(\Ombar\times [τ,T])}{\ue}\le \gc{uholder}(τ,T).
 \]
\end{lemma}
\begin{proof}
 On account of Lemmata~\ref{lem:uinfty} and \ref{lem:navq}, another classical parabolic regularity result (see e.g. \cite[Theorem 4]{diBenedetto}) is applicable and readily yields these Hölder bounds.
\end{proof}

\begin{lemma}\label{lem:uC2}
 For every $T>0$ and $τ\in(0,T)$ there are $α>0$, $\newgc{vC2}(τ,T)>0$ and $\newgc{uC2}(τ,T)>0$ such that, for all $ε>0$,  
 \begin{align*}
  &\normm{C^{2+α,1+\f{α}2}(\Ombar\times [τ,T])}{\ve}\le \gc{vC2}(τ,T),\\
  &\normm{C^{2+α,1+\f{α}2}(\Ombar\times [τ,T])}{\ue}\le \gc{uC2}(τ,T).
 \end{align*}
\end{lemma}
\begin{proof}
Classical Schauder estimates for parabolic equations (like \cite[Thm. 4.31]{lieberman_parabolicbook} -- or, more precisely, a time-local version thereof, i.e. said theorem is applied to $ζ\ve$ or $ζ\ue$, where $ζ$ is a smooth temporal cutoff function with $ζ=0$ at $t=0$) relying on the outcome of Lemma~\ref{lem:ualpha} can firstly be applied to $\ve$ and then to $\ue$, resulting in the above statement.
\end{proof}

\section{Proof of Thm. \ref{thm:main}: Convergence and continuity at zero}
\begin{lemma}\label{lem:conv}
 There are functions $u,v\colon\Om\times[0,∞)\to [0,∞)$ and a sequence $(ε_j)_{j\inℕ}$ such that $ε_j\to 0$ as $j\to \infty$ and 
 \begin{align}
  \une&\to u(\cdot,0) \qquad &&\text{in } \Lom1,
  \label{conv:uL1}\\
  \ue&\to u \qquad &&\text{in } C^{2,1}(\Ombar\times(0,∞)),\label{conv:uC2}\\
  \vne&\to v(\cdot,0) \qquad &&\text{in } 
  \Wom12\label{conv:vW12}\\
  \ve&\to v \qquad &&\text{in } C^{2,1}(\Ombar\times(0,∞)).\label{conv:vC2}
 \end{align}
 as $ε=ε_j\searrow 0$.
\end{lemma}
\begin{proof}
 While convergence at $t=0$, that is, \eqref{conv:uL1} and \eqref{conv:vW12}, has been prescribed in \eqref{init-eps:conv:u} and \eqref{init-eps:conv:v}, \eqref{conv:uC2} and \eqref{conv:vC2} along a suitable sequence are easily obtained from Lemma~\ref{lem:uC2}.
\end{proof}

\begin{lemma}\label{lem:ucts}
 The function $u$ from Lemma~\ref{lem:conv} belongs to $C_w([0,∞);\Lom1)$ with $u(\cdot,0)=u_0$.
\end{lemma}
\begin{proof}
 In light of \eqref{conv:uC2} we only have to show continuity of $u$ at $t=0$. Let us assume there were one sequence $(t_k)_{k\inℕ}\searrow 0$ such that $u(\cdot,t_{k_j})\not\wto u(\cdot,0)=u_0$ in the sense of weak $\Lom1$-convergence for every subsequence $(t_{k_j})_{j\inℕ}$ of $(t_k)_{k\inℕ}$. Combining \eqref{conv:uC2} with Lemma~\ref{lem:Phi-time} we then could see that $\io Φ(u(\cdot,t_k))\le \gc{Phi}$ for all $k\inℕ$ and hence $\set{u(\cdot,t_k)|k\inℕ}$ were equiintegrable and we could find a subsequence $(t_{k_l})_{l\inℕ}$ of $(t_k)_{k\inℕ}$ such that $u(\cdot,t_{k_l})$ converged to some $\tilde u\in \Lom1$ as $l\to\infty$, weakly in $\Lom1$, by the Dunford--Pettis theorem (\cite[Thm. IV.8.9]{dunford-schwartz}). But Lemma~\ref{lem:ue-at0} together with \eqref{conv:uC2} would ensure that $\tilde u=u_0$, in contradiction to the choice of $(t_k)_{k\inℕ}$.
\end{proof}

\begin{lemma}\label{lem:vcts}
 The function $v$ from Lemma~\ref{lem:conv} is continuous at $t=0$ as a function with values in $\Lom2$ (with respect to the norm topology) and as function with values in $\Wom12$ equipped with the weak topology and satisfies $v(\cdot,0)=v_0$.
\end{lemma}
\begin{proof}
If this were not the case, we could find a sequence $(t_k)_{k\inℕ}$ such that $t_k\to 0$ and none of the subsequences of $(v(\cdot,t_k))_{k\inℕ}$ converged to $v_0$ (in $\Lom2$ or weakly in $\Wom12$). The bounds in Lemmata~\ref{lem:vl2bound} and \ref{lem:vderivativebounds} together with \eqref{conv:vC2}, however, would ensure boundedness of $(v(t_k))_{k\inℕ}$ in $\Wom12$ and hence convergence to some $\tilde{v}$ along a subsequence, in both $\Lom2$ and weakly in $\Wom12$. 
Because in
\begin{align}\label{vatzero}
\norm[\Lom2]{v(t_k)-v_0}&\le \norm[\Lom2]{v(t_k)-\ve(t_k)}+\norm[\Lom2]{\ve(t_k)-\vne} + \norm[\Lom2]{\vne-v_0} 
\end{align}
we could use Hölder's inequality and Lemma~\ref{lem:vt} to estimate 
\[
 \norm[\Lom2]{\ve(t_k)-\vne}
 \le \sqrt{t_k} \kl{\int_0^{t_k} \norm[\Lom2]{\vet}}^2 \le \sqrt{t_k}\sqrt{\gc{vt}(1+t_k)}\quad \text{for every } k\in ℕ,
\]
independently of $ε>0$, \eqref{init-eps:conv:v} and \eqref{conv:vC2} in conjunction with \eqref{vatzero} would finally serve to identify $\tilde v=v_0$, in contradiction to the choice of $(t_k)_{k\inℕ}$. 
\end{proof}

\begin{proof}[Proof of Thm.~\ref{thm:main}]
 By Lemma~\ref{lem:conv}, along a suitable subsequence, the approximate solutions provided by Lemma~\ref{lem:ex} converge to some pair $(u,v)$ of functions which belong to $C^{2,1}(\Ombar\times(0,∞))$ and solve \eqref{thesystem:u}, \eqref{thesystem:v} and \eqref{thesystem:bdry} by \eqref{conv:uC2} and \eqref{conv:vC2}. That they solve \eqref{thesystem:initial} and belong to $C_w([0,∞);\Lom1)$ follows from Lemma~\ref{lem:ucts} for $u$ and Lemma~\ref{lem:vcts} in case of $v$ (where the statement is even slightly stronger).
\end{proof}

{
\footnotesize
\setlength{\parskip}{0pt}
\setlength{\itemsep}{0pt plus 0.3ex}

\def\cprime{$'$}

}

\end{document}